\documentclass{amsart}
\usepackage{amssymb}
\usepackage{graphicx}
\usepackage{amsrefs}

\copyrightinfo{2008}{Radoslaw K. Wojciechowski}

\newtheorem{theorem}{Theorem}[section]
\newtheorem{lemma}[theorem]{Lemma}
\newtheorem{proposition}[theorem]{Proposition}
\newtheorem{corollary}[theorem]{Corollary}

\theoremstyle{definition}
\newtheorem{definition}[theorem]{Definition}

\theoremstyle{remark}
\newtheorem{remark}[theorem]{Remark}

\numberwithin{equation}{section}

\def \la{\langle}
\def \ra{\rangle}
\def \pdt{\frac{\partial}{\partial t}}
\def \pdtu{\frac{\partial u}{\partial t}}

\def \pdtw{\frac{\partial w}{\partial t}}
\def \ov{\overline }
\def \un{\underline }

\begin{document}

\title{Heat Kernel and Essential Spectrum of Infinite Graphs}
\author{Rados{\l}aw K. Wojciechowski}
\address{Graduate Center of the City University of New York, 365 Fifth Avenue, New York, NY, 10016.}
\email{rwojciechowski@gc.cuny.edu}
\date{January 2008}

\begin{abstract}
We study the existence and uniqueness of the heat kernel on infinite, locally finite, connected graphs.  For general graphs, a uniqueness criterion, shown to be optimal, is given in terms of the maximal valence on spheres about a fixed vertex.  A sufficient condition for non-uniqueness is also presented.  Furthermore, we give a lower bound on the bottom of the spectrum of the discrete Laplacian and use this bound to give a condition ensuring that the essential spectrum of the Laplacian is empty.
\end{abstract}

\maketitle

\section{Introduction}
The first part of this paper investigates the stochastic completeness of the heat kernel on infinite, locally finite, connected graphs.  The heat kernel considered here is a real-valued function of a pair of vertices and a continuous time parameter and is the smallest non-negative fundamental solution for the discrete heat equation.  The second section of this paper outlines a construction of the heat kernel using an exhaustion argument.  Since this construction is well-known in the case of Riemannian manifolds \cite{Dod83} and was written up in full detail in \cites{Woj07, Web08} we omit some details.  The heat kernel will  generate a bounded solution of the heat equation given a bounded initial condition.  Stochastic completeness of the heat kernel is equivalent to the uniqueness of this solution.

In the third section of this paper we state a definition of stochastic incompleteness and give several equivalent conditions as presented in \cite{Gri99}.  We use one of these conditions, specifically, the existence of a bounded, positive, $\lambda$-harmonic function for a negative constant $\lambda$, to establish a criterion ensuring the stochastic completeness of the heat kernel on a general graph in terms of the maximum valence on spheres about a fixed vertex.  This result should be compared with a result of Dodziuk and Mathai \cites{DodMat06, Dod07} which gives uniqueness of bounded solutions of the heat equation for graphs of bounded valence and a result in \cite{Web08} which gives uniqueness under an assumption on the curvature of the graph.

We then give a similar condition implying the stochastic incompleteness the heat kernel on a general graph.  Specifically, we show that if the minimum number of edges leading away from a fixed vertex on the graph increases uniformly at a sufficient rate on spheres of increasing radii then the graph will be stochastically incomplete.  In fact, we show that it is sufficient that this condition holds in a subgraph of the entire graph, provided that the subgraph is connected to its complement at a single vertex.

We then show that our characterizations of stochastic completeness in terms of the minimum and maximum valence of vertices on spheres about a fixed vertex are optimal by introducing a family of trees we call \emph{model}.  By definition, these trees contain a vertex, here referred to as the \emph{root} for the model, such that the valence at every other vertex depends only on the distance from the root.  These trees are also sometimes called \emph{symmetric about the root} or \emph{radially symmetric} with the \emph{branching number} being the common valence on spheres \cite{Fuj96} but we call them model because they are the analogues of rotationally symmetric or model Riemannian manifolds \cites{Gri99, GreWu79}.  In the case of model trees the sufficient conditions mentioned above are also necessary.  In particular, these trees offer examples of infinite, stochastically incomplete graphs, that is, ones for which bounded solutions of the heat equation are not uniquely determined by initial conditions.  We also prove two inequalities comparing the heat kernel on a model tree and the heat kernel on a general graph.  These inequalities were inspired by an analogous result of Cheeger and Yau on Riemannian manifolds \cite{CheYau81}.

In the final part of this paper we study the spectrum of the Laplacian on a general graph.  We use a characterization of $\lambda_0(\Delta)$, the bottom of the spectrum of the Laplacian, in terms of Cheeger's constant established in \cites{Dod84, DodKen86} to give a lower bound for $\lambda_0(\Delta)$ under a curvature assumption on the graph.  Using this lower bound and the fact that, when Dirichlet boundary conditions are imposed, the essential spectrum of the Laplacian on the entire graph is the same as the essential spectrum of the graph with a finite subgraph removed \cites{DonLi79, Fuj96}, we then show that the essential spectrum is empty provided that the graph is branching rapidly and the curvature condition is satisfied.  This result should be compared with \cite{Kel07} where, for a graph whose Cheeger constant at infinity is positive, rapid branching is shown to be necessary and sufficient for the essential spectrum to be empty.

\subsection{Acknowledgments}
This paper is based on the author's $\textrm{Ph.D.}$ thesis \cite{Woj07} written under the guidance of J{\'o}zef Dodziuk.  It is a pleasure to thank Professor Dodziuk for his constant support and assistance while this work was carried out.  I would also like to thank the referee for many useful comments which greatly improved the exposition of the paper and Matthias Keller for stimulating discussions which led to improvement of results.

\section{The Heat Kernel}
\subsection{Preliminaries}
We now establish our notation and several basic lemmas which will be used throughout.  $G=(V(G),E(G))$ will denote an infinite, locally finite, connected graph without loops or multiple edges where $V=V(G)$ is the set of vertices of $G$ and $E=E(G)$ is the set of edges.  We will write $x \in G$ when $x$ is a vertex of $G$.  We use the notation $x \sim y$ to indicate that an edge connects vertices $x$ and $y$ while $[x,y]$ will denote an \emph{oriented} edge with initial vertex $x$ and terminal vertex $y$.  At times, to be able to write down certain formulas unambiguously, we assume that our graphs are oriented, that is, that every edge has a chosen, fixed orientation, but none of our results depend on the choice of this orientation.

We use the notation $m(x)$ to indicate the \emph{valence} at a vertex $x$, that is, $m(x)$ denotes the number of vertices connected by an edge to $x$. For any two vertices $x$ and $y$, $d(x,y)$, the \emph{distance} between $x$ and $y$, is the number of edges in the shortest path connecting $x$ and $y$.  We let $r(x)=d(x,x_0)$ denote the distance between $x$ and a fixed vertex $x_0$.  By a \emph{function on the graph} we mean a mapping $f: V \to \mathbf{R}.$  We will denote the set of all such functions by $C(V)$.  $C_0(V)$ will denote the set of all functions on $G$ with finite support. $C_0(V)$ is dense in $\ell^2(V)$, the space of all \emph{square summable} functions on $G$,
\[ \ell^2(V) = \Big \{ f:V \to \mathbf{R} \  \Big| \ \sum_{x \in V} f(x)^2 < \infty  \Big \} \]
which is a Hilbert space with inner product
\[ \la f, g \ra = \sum_{x \in V} f(x) g(x). \]
Similarly, we let $\ell^2(\tilde{E})$ denote the Hilbert space of all square summable functions on $\tilde{E}$, the set of all oriented edges of $G$,
\[ \ell^2(\tilde{E}) =  \Big\{ \varphi: \tilde{E} \to \mathbf{R} \ \Big| \sum_{[x,y] \in \tilde{E}} \varphi([x,y])^2 < \infty  \Big \} \]
with a similarly defined inner product.

If $f$ is a function on the vertices of $G$ then $d$, the \emph{coboundary operator}, maps $f$ to a function on oriented edges: $df([x,y]) = f(y) - f(x)$.  The \emph{Laplacian} $\Delta$ is defined by the formula
\begin{equation}\label{Laplacian}
\Delta f(x) = \sum_{y \sim x} \big(f(x) - f(y) \big) = m(x)f(x) - \sum_{y \sim x} f(y)
\end{equation}
where the summation runs over all neighbors of $x$.  It follows immediately from (\ref{Laplacian}) that the Laplacian is a bounded operator on $\ell^2(V)$ if and only if the valence of $G$ is bounded, that is, $m(x) \leq M$ for all vertices $x$.

If $D$ is a finite subgraph of $G$ then int $D$ will denote the \emph{interior} of $D$ which consists of vertices of $D$ all of whose neighbors are in $D$, that is, $\textrm{int } D = \{x \ | \ x \in D \textrm{ and if } y \sim x \textrm{ then } y \in D \}.$  The complementary set of vertices in $D$ is called the \emph{boundary} of $D$, denoted by $\partial D$, so that, $\partial D = \{x \ | \ x \in D \textrm{ and there exists } y \sim x \textrm{ such that } y \not \in D \}$.  An easy calculation gives the following analogue of Green's Theorem:

\begin{lemma}\label{Green's}  If $D$ is a finite subgraph of $G$ then for $f, g \in C(V)$
\[ \sum_{x \in D} \Delta f(x) g(x) = \sum_{[x,y] \in \tilde{E}(D)} df([x,y]) dg([x,y]) +  \sum_{\substack{x \in \partial D \\ z \sim x, z \not \in D}} \big(f(x) - f(z)\big) g(x). \]
\end{lemma}
If $g$ vanishes on the boundary of $D$ then we may write this as $\la \Delta f, g \ra_{V(D)} = \la df, dg \ra_{\tilde{E}(D)}$.  Furthermore, if $f$ or $g$ is finitely supported then
\[ \la \Delta f, g \ra = \la df, dg \ra. \]

A function $u: V \times [0,\infty) \to \mathbf{R}$ is said to be a solution of the \emph{heat equation} if it is differentiable in the second parameter and satisfies
\[ \Delta u(x,t) + \pdtu (x,t) = 0 \]
for all vertices $x$ and all $t \geq 0$.  As in the smooth case we have the following maximum principles:

\begin{lemma}\label{maxheatlem1}
Suppose that $D$ is a finite, connected subgraph of $G$ and that the function $u: D \times [0,T] \to \mathbf{R}$ is differentiable in the second parameter and satisfies $\Delta u + \pdtu \leq 0 \ on \ \textup{int } D \times [0,T].$  Then
\[ \max_{D \times [0,T]} \ u = \max_{\substack{D \times \{0\} \ \cup\\ \partial D \times [0,T]}} \ u. \]
If $u$ satisfies $\Delta u + \pdtu \geq 0 \ on \ \textup{int } D \times [0,T]$ then
\[ \min_{D \times [0,T]} \ u = \min_{\substack{D \times \{0\} \ \cup\\ \partial D \times [0,T]}} \ u. \]
Therefore, if $u$ satisfies the heat equation on \textup{int} $D \times [0,T]$ then it attains both its maximum and its minimum on the parabolic boundary  $(D \times \{0\}) \ \cup \ (\partial D \times [0,T])$.
\end{lemma}
\begin{proof}
We give a proof when $u$ satisfies $\Delta u + \pdtu \leq 0$ and note that the second statement of the lemma follows by considering $-u$.  Let $w= u - \epsilon t$ for $\epsilon > 0$ so that $\Delta w + \pdtw \leq -\epsilon < 0.$  If $w$ has a maximum at $(\hat{x},t_0) \in \textrm{int } D \times (0,T]$ then $\pdtw (\hat{x},t_0) \geq 0$ and $\Delta w(\hat{x},t_0) = \sum_{x \sim \hat{x}}\big( w(\hat{x},t_0) - w(x,t_0) \big) \geq 0$ yielding a contradiction.  Therefore, $w$ attains its maximum on $(D \times \{0\})  \cup (\partial D \times [0,T]).$
It follows that
\begin{eqnarray*}
\max_{D \times [0,T]} \ u &=&  \max_{D \times [0,T]} ( w + \epsilon t ) \ \leq \  \max_{D \times [0,T]} \ w + \epsilon T \\
&=&  \max_{\substack{D \times \{0\} \ \cup\\ \partial D \times [0,T]}} \ w + \epsilon T  \ \leq \  \max_{\substack{D \times \{0\} \ \cup\\ \partial D \times [0,T]}} \ u + \epsilon T .
\end{eqnarray*}
Since $\epsilon$ was arbitrary this implies that
\[ \max_{D \times [0,T]} \ u = \max_{\substack{D \times \{0\} \ \cup\\
\partial D \times [0,T]}} \ u.  \]
\end{proof}

\begin{lemma}\label{maxheatlem2}
Suppose that $D$ is a finite, connected subgraph of $G$ and that $u$ satisfies $\Delta u + \pdtu = 0$ on \textup{int }$D \times [0,T].$  If there exists $(\hat{x},t_0) \in \textup{int} \ D \times (0,T)$ such that $(\hat{x},t_0)$ is a maximum (or a minimum) for $u$ on $D \times [0,T]$ then $u(x,t_0) = u(\hat{x},t_0)$ for all $x \in D$.
\end{lemma}
\begin{proof}
At either a maximum or minimum $\pdtu (\hat{x},t_0) = 0$ implying that $\Delta u(\hat{x},t_0) = \sum_{x \sim \hat{x}}\big(u(\hat{x},t_0) - u(x,t_0) \big) = 0.$ In either case, this implies that $u(x,t_0) = u(\hat{x},t_0)$ for all  $x \sim \hat{x}$. Iterating the argument and using the assumption that $D$ is connected gives the statement of the lemma.
\end{proof}

\subsection{Construction of the Heat Kernel}
As mentioned in the introduction, the construction given here follows the one on open manifolds presented in \cite{Dod83}*{Section 3} as we exhaust the graph by finite, connected subgraphs, recall a definition of the heat kernel with Dirichlet boundary conditions for each subgraph in the exhaustion, and then pass to the limit.  To make our objective precise, for $t\geq 0$ and vertices $x$ and $y$, $p_t(x,y)$, the \emph{heat kernel}, will be the smallest non-negative function smooth in $t$, satisfying the heat equation $\Delta p_t(x,y) + \pdt p_t(x,y) = 0$ in either $x$ or $y$, and satisfying $p_0(x,y) = \delta_{x}(y)$ where $\delta_{x}$ is the delta function at the vertex $x$.  The heat kernel will generate a bounded solution of the heat equation on $G$ for any bounded initial condition.  Precisely, for a function $u_0$ bounded on the vertices of $G$, $u(x,t) = \sum_{y \in V} p_t(x,y)u_0(y)$ is bounded, smooth for $t\geq0$, and satisfies
\begin{equation}\label{heatequation}
 \left\{  \begin{array}{ll}
 \Delta u(x,t) + \pdtu (x,t) = 0 & \textrm{ for } x \in V, \  t\geq0 \\
 u(x,0) = u_0(x) & \textrm{ for } x \in V.
 \end{array} \right.
\end{equation}

To begin the construction fix a vertex $x_0 \in V$ and let $B_r=B_r(x_0)$ denote the connected subgraph of $G$ consisting of those vertices in $G$ that are at most distance $r$ from $x_0$ and all edges of $G$ that such vertices span.  It follows that $B_r \subseteq B_{r+1}$ and $G= \cup_{r=0}^{\infty} B_r$.  We let $C(B_r, \partial B_r) = \{ f \in C(B_r) \ | \ f_{| \partial B_r} = 0 \}$ denote the set of functions on $B_r$ which vanish on $\partial B_r$ and let $\Delta_r$ denote the \emph{reduced Laplacian} which acts on this space  by
\[ \Delta_r f(x) = \left \{ \begin{array}{ll}
 \Delta f(x) & \textrm{for } x \in \textrm{int } B_r \\
 0 & \textrm{otherwise.}
 \end{array} \right. \]
By Lemma \ref{maxheatlem1} solutions of the heat equation on $B_r$ with Dirichlet boundary conditions are uniquely determined by initial data.  Furthermore, by applying Lemmas \ref{Green's} and \ref{maxheatlem2}, it follows that $\Delta_r$ has a finite set of real eigenvalues $0 < \lambda_0^r < \lambda_1^r \leq \ldots \leq \lambda_{k(r)}^r$ and a corresponding set of eigenfunctions  $\{ \phi_i^r\}_{i=0}^{k(r)}$ which form an orthonormal basis for $C(B_r, \partial B_r)$.    Since $\Delta_r$ acts on the finite dimensional vector space $C(B_r, \partial B_r)$ we can consider the operator $e^{-t\Delta_r} = I - t\Delta_r + \frac{t^2}{2}\Delta_r^2 - \frac{t^3}{6}\Delta_r^3 + \dots$ and let $p_t^r(x,y)$ denote its kernel.
\begin{definition}\label{Dirichletheatkernelsdefinition}
For vertices $x$ and $y$ in $B_r$ and $t\geq 0$ 
\[ p_t^r(x,y) = \la e^{-t\Delta_r}\delta_x, \delta_y \ra =  e^{-t\Delta_r}\delta_x(y). \]
\end{definition}

\begin{proposition} \label{Dirichletheatkernels}
$p_t^r(x,y)$ satisfies the following properties:
\begin{enumerate}
\item[1)]  $p_t^r(x,y) = p_t^r(y,x)$ and  $p_t^r(x,y) = 0$ if $x \in \partial B_r$ or $y \in \partial B_r$.
\item[2)]  $\Delta_r p_t^r(x,y) + \pdt p_t^r(x,y) = 0$ where $\Delta_r$ denotes the reduced Laplacian applied in either $x$ or $y$.
\item[3)] $p_{s+t}^{r}(x,y) = \sum_{z \in B_r}p_s^r(x,z) p_t^r(z,y)$.
\item[4)] $p_0^r(x,y) = \delta_x(y)$ \ for $x,y \in \textup{int } B_r$.
\item[5)]  $p_t^r(x,y) = \sum_{i=0}^{k(r)} e^{-\lambda_i^r t} \phi_i^r (x) \phi_i^r (y)$.
\item[6)] $p_t^r(x,y) > 0 \ \textrm{ for } t > 0, x,y \in \textup {int } B_r$.
\item[7)] $\sum_{y \in B_r} p_t^r(x,y) < 1  \  \textrm{ for } t > 0, x \in B_r$. \newline
\end{enumerate}
\end{proposition}

\begin{proof}
1), 2), 3), and 4) are clear from definitions and from the fact that $\Delta_r$ is symmetric.  For 5), let $q_t^r(x,y) = \sum_{i=0}^{k(r)} e^{-\lambda_i^r t} \phi_i^r (x) \phi_i^r (y)$ and observe that, since $\{ \phi_i^r\}_{i=0}^{k(r)}$ forms an orthonormal basis for $C(B_r,\partial B_r)$, $q_t^r(x,y)$ satisfies $q_0^r(x,y) = \delta_x(y)$.  Hence, $p_t^r(x,y)$ and $q_t^r(x,y)$ are two functions satisfying the heat equation on the interior of $B_r$, both vanish on the boundary $\partial B_r$ and satisfy the same initial condition.  Therefore, by Lemma \ref{maxheatlem1}, $p_t^r(x,y) = q_t^r(x,y)$.  For 6), Lemma \ref{maxheatlem1} implies that $0 \leq p_t^r(x,y) \leq 1$ for all vertices $x$ and $y$ and $t \geq 0$.  If there exist $\hat{x}$ and $\hat{y}$ in the interior of $B_r$ and $t_0>0$ such that $p^r_{t_0}(\hat{x},\hat{y})=0$ then, by Lemma \ref{maxheatlem2}, since $p_t^r(x,y)$ satisfies the heat equation in both $x$ and $y$, it would follow that $p_{t_0}^r(x,y) = p_{t_0}^r(\hat{x},\hat{y}) = 0$ for all $x$ and $y$.  In particular, $p_{t_0}^r(x,x)=0$ implying, from 5), that $\phi_i^r(x)=0$ for all $i$ and all $x$ in the interior of $B_r$ contradicting the fact that $\{ \phi_i^r\}_{i=0}^{k(r)}$ is, by definition, an orthonormal basis for $C(B_r, \partial B_r)$.

To prove 7), we first note that, by 4), $\sum_{y \in B_r} p_0^r(x,y) = 1$ for $x \in \ $int $B_r$.  Then, applying Lemma \ref{Green's} and letting $\Delta_y$ denote the Laplacian applied in $y$, it follows that
\begin{eqnarray*}
\pdt \sum_{y \in \textup{int } B_r} p_t^r(x,y) &=& \sum_{y \in \textup{int }B_r} - \Delta_y p_t^r(x,y) \\
&=& \sum_{\substack{y \in \textup {int } B_r \\ z \sim y, \ z \in \partial B_r}} \Big( p_t^r(x,z) - p_t^r(x,y) \Big) < 0
\end{eqnarray*}
proving the claim.
\end{proof}

Extend each $p_t^r(x,y)$ to be defined for all vertices $x$ and $y$ in $G$ by letting it be $0$ for vertices outside of $B_r$.  Since $B_r \subseteq B_{r+1}$ it follows by applying Lemma \ref{maxheatlem1} that $p^r_t(x,y) \leq p^{r+1}_t(x,y)$ and, combining statements 6) and 7) from Proposition \ref{Dirichletheatkernels}, that $0\leq p_t^r(x,y) \leq 1.$  Therefore, the sequence of heat kernels $p_t^r(x,y)$ converges as $r \to \infty$ and the heat kernel $p_t(x,y)$ can be defined as the limit.
\begin{definition}
\[ p_t(x,y) = \lim_{r \to \infty} p_t^r(x,y) \textrm{ for  vertices } x \textrm{ and } y \textrm{ in } V \textrm{ and }  t \geq 0. \]
\end{definition}

Dini's Theorem implies that the convergence is uniform in $t$ on every compact subset of $[0, \infty)$.  To show that $p_t(x,y)$ is differentiable and satisfies the heat equation observe that $\pdt p_t^r(x,y) = -\Delta_{r,x} p_t^r(x,y) = \sum_{z \sim x} \big(p_t^r(z,y) - p_t^r(x,y) \big) \to -\Delta p_t(x,y)$ where the convergence is uniform on compact subsets of $[0,\infty)$.  Therefore, $p_t(x,y)$ is differentiable and satisfies the heat equation in both $x$ and $y$.  In fact, iterating this argument shows that $p_t(x,y)$ is smooth in $t$ on compact subsets of $[0,\infty)$.  From this argument and the corresponding properties of the heat kernels $p_t^r(x,y)$ statements 1) through 6) of the following theorem now follow.
\begin{theorem}\label{heatkernel}
$p=p_t(x,y)$ satisfies
\begin{enumerate}
\item[1)]  $p_t(x,y) = p_t(y,x)$ and $p_t(x,y) > 0$ for $t > 0$.
\item[2)]  $p$ is $C^{\infty}$ on $[0,\infty)$.
\item[3)]  $\Delta p_t(x,y) + \pdt p_t(x,y) = 0$ where $\Delta$ denotes the Laplacian applied in either $x$ or $y$.
\item[4)]  $p_{s+t}(x,y) = \sum_{z \in V}p_s(x,z) p_t(z,y)$.
\item[5)]  $\sum_{y \in V} p_t(x,y) \leq 1$ for $t \geq 0$, $x \in V$.
\item[6)]  $\lim_{t \to 0} \sum_{y \in V} p_t(x,y) = p_0(x,y) = \delta_x(y)$.
\item[7)]  $p$ is independent of the exhaustion used to define it.
\item[8)]  $p$ is the smallest non-negative function that satisfies Properties 3) and 6).
\end{enumerate}
\end{theorem}
\begin{proof}
For 6), observe that $p_t^r(x,x) \leq p_t(x,x) \leq 1$ and $\sum_{y \in V}p_t(x,y) = p_t(x,x) + \sum_{y \in V, y \not = x} p_t(x,y) \leq 1$ and let $t \to 0$.  7) and 8) follow by applying Lemma \ref{maxheatlem1}.
\end{proof}

\begin{remark}
If $u_0$ denotes a bounded function on $G$ then, by combining the properties above, $u(x,t) = P_t u_0(x) =  \sum_{y \in V} p_t(x,y)u_0(y)$ is a bounded function, differentiable in $t$, continuous for $t\geq0$, satisfying (\ref{heatequation}).
\end{remark}

\begin{remark}
Letting $P_t u_0(x) = \sum_{y \in V}p_t(x,y) u_0(y)$ for $u_0 \in C_0(V)$, it is also true, as in \cite{Dod83}*{Proposition 4.5}, that $P_t u_0(x) = e^{-t \tilde{\Delta}} u_0(x)$ where $\tilde{\Delta}$ is the unique self-adjoint extension of $\Delta$ with domain $C_0(V)$ in $\ell^2(V)$.  The proof of this fact, and of the essential self-adjointness of the Laplacian, can be found in \cites{Web08, Woj07}.
\end{remark}

\section{Stochastic Incompleteness}

\subsection{Stochastic Incompleteness}
Let $\mathbf{1}$ denote the function which is 1 at every vertex of $G$.  From Property 5) of Theorem \ref{heatkernel} we know that, for every vertex $x$ and every $t \geq 0$, $P_t \mathbf{1}(x) = \sum_{y \in V} p_t(x,y) \leq 1.$
\begin{definition}
$G$ is \emph{stochastically incomplete} if for some vertex $x$ and some $t>0$ the heat kernel $p_t(x,y)$ on $G$ satisfies $\sum_{y \in V} p_t(x,y) < 1$.
\end{definition}

In the next theorem we single out several conditions which are equivalent to stochastic incompleteness.  In particular, stochastic incompleteness is equivalent to the non-uniqueness of bounded solutions of the heat equation.
\begin{theorem} \cite{Gri99}*{Theorem 6.2} \label{stochasticincomp}  The following statements are equivalent:
\begin{enumerate}
\item[1)] For some $t>0,$ some $x \in V$, $P_t \mathbf{1}(x) < 1.$
\item[$1')$] For every $t>0$, every $x \in V$, $P_t \mathbf{1}(x) < 1.$
\item[2)] For every $\lambda < 0$ there exists a positive (equivalently, non-zero) bounded function $v$ on $G$ satisfying $\Delta v = \lambda v$.
\item[$2')$] For every $\lambda < 0$ there exists a non-negative, non-zero, bounded function $v$ on $G$ satisfying $\Delta v \leq \lambda v$.
\item[3)]  There exists a non-zero, bounded function $u$ satisfying
\[ \left\{  \begin{array}{ll} \Delta u(x,t) + \pdtu(x,t) = 0 & \textrm{ for } x \in V, \ t\geq0\\
 u(x,0) = 0 & \textrm{ for }  x \in V.
 \end{array} \right.  \]
\end{enumerate}
\end{theorem}
\begin{proof}
That $1')$ implies 1) is obvious.  To show that 1) implies $1')$ note that if there exists a $t_0$ such that $P_{t_0} \mathbf{1} = \mathbf{1}$ then by the semi-group property, part 4) in Theorem \ref{heatkernel}, it follows that $P_t \mathbf{1}=\mathbf{1}$ for all $t$.  Similarly, if for some $x$, $P_t \mathbf{1}(x)=1$ then the proof of Lemma \ref{maxheatlem2} will imply that $P_t \mathbf{1}(x)=1$ for all vertices $x$.

To show that $1')$ implies 2), let $u(x,t)=P_t \mathbf{1}(x) < 1$ and $w(x)=\int_0^{\infty}e^{\lambda t} u(x,t) dt$ for $\lambda<0$.  Estimating gives $0 < w <-\frac {1}{\lambda}$ while integration by parts yields $\Delta w=1 + \lambda w.$  Letting $v=1+\lambda w$ it follows that $v$ satisfies $0<v<1$ and $\Delta v = \lambda v$.

To show that $2')$ implies 2) we fix a vertex $x_0$ and let $B_r=B_r(x_0)$ as in the construction of the heat kernel so that $B_r \subseteq B_{r+1}$ and $G=\cup_{r=0}^\infty B_r.$  We first show that, for every $\lambda<0$, there exists a unique function $v_r$ satisfying
\begin{equation}\label{fredholm}
 \left\{  \begin{array}{ll} \Delta v_r(x) - \lambda v_r(x) = 0 & \textrm{for $x \in$ int } B_r\\
 {v_r}_{| \partial B_r} = 1.
 \end{array} \right.
\end{equation}
This follows since the associated homogeneous system obtained by replacing ${v_r}_{| \partial B_r} = 1$ with ${v_r}_{| \partial B_r} = 0$ in (\ref{fredholm}) has only trivial solutions.  To see this, let $w$ be a solution of the homogeneous system and suppose that $w$ is non-zero.  We may then assume that there exists $\hat{x} \in$ int $B_r$ such that $w(\hat{x}) >0 $ and $\hat{x}$ is a maximum for $w$ on the interior of $B_r$.  This implies that $\Delta w(\hat{x}) \geq 0$ while $\Delta w(\hat{x}) = \lambda w(\hat{x}) <0$.  The contradiction establishes the existence and uniqueness of a solution to (\ref{fredholm}).

Let, therefore, $v_r$ denote the unique solution of $(\ref{fredholm})$.  We claim that $0<v_r<1$ on the interior of $B_r$ and that, if we extend each $v_r$ to be 1 outside of $B_r$, then $v_r \geq v_{r+1}$.  To show that $v_r>0$ assume that there exists an $\hat{x}$ in the interior of $B_r$ such that $v_r(\hat{x}) \leq 0$.  We may assume that $\hat{x}$ is a minimum for $v_r$ from which it follows that $\Delta v_r(\hat{x}) \leq 0$.  On the other hand, $\Delta v_r(\hat{x})=\lambda v_r(\hat{x}) \geq 0$ which implies that $v_r(x)=v_r(\hat{x})$ for all $x$ connected to $\hat{x}$.  Iterating this argument gives a contradiction since ${v_r}_{|\partial B_r} =1$.  The other claims are proved in a similar manner.  Therefore, $\{ v_r \}_{r=0}^{\infty}$ is a non-increasing sequence of bounded functions and we let $\lim_{r \to \infty} v_r = v$ with $0 \leq v \leq 1$ and $\Delta v = \lambda v$.   It remains to show that $v$ is non-zero.  By arguments given below this will imply that there then exists a positive, bounded function satisfying $\Delta v = \lambda v$.   Now, by assumption, there exists a bounded function $w$ such that $\Delta w \leq \lambda w$ and $w$ is non-negative and non-zero.  Assuming that $|w|<1$, it can be shown that $v_r \geq w$ by a maximum principle argument as used above.  It follows that $v$ is non-zero by letting $r \to \infty$.

To show that 2) implies 3) suppose that $v$ is a non-zero, bounded function satisfying $\Delta v = \lambda v$ for $\lambda <0$.  It follows that, $e^{-\lambda t} v$ and $P_t v$ are both bounded solutions of
\[ \left\{  \begin{array}{ll} \Delta u(x,t) + \pdtu(x,t) = 0 & \textrm{for } (x,t) \in V \times [0,T]\\
 u(x,0) = v(x) & \textrm{for } x \in V
 \end{array} \right. \]
and, for $t>0$, $|e^{-\lambda t} v(x)|>|v(x)|$ while $|P_t v(x)| \leq \sup_{y \in V} |v(y)|$.  Therefore, by considering the difference of $e^{-\lambda t} v$ and $P_t v$, there exists a non-zero, bounded solution of the heat equation with initial condition equal to 0 for a bounded time interval.  By the argument given below this is enough to imply condition 1) which is equivalent to $1')$ and given $1')$, that is, $P_t \mathbf{1} < \mathbf{1}$, it follows that $\mathbf{1} - P_t \mathbf{1}$ will give a non-zero, bounded solution of the heat equation with zero initial condition for all time thereby completing the proof.

To show that 3) implies 1) suppose that $u(x,t)$ satisfies the conditions in 3).  We may assume, by rescaling if necessary, that $|u| < 1$ and that there exists a vertex $\hat{x}$ and $t_0>0$ such that $u(\hat{x},t_0)>0$.  It follows that $w(x,t) = 1 - u(x,t)$ is bounded, positive, and satisfies
\begin{equation} \label{heat1}
\left\{  \begin{array}{ll} \Delta w(x,t) + \pdtw(x,t) = 0 & \textrm{for }(x,t) \in V \times [0,T]\\
 w(x,0) = 1 & \textrm {for } x \in V
 \end{array} \right.
\end{equation}
with $w(\hat{x},t_0) < 1$.  Since $P_t \mathbf{1}$ is, by construction, the smallest positive solution to (\ref{heat1}) it follows that $P_{t_0} \mathbf{1}(\hat{x}) \leq w(\hat{x},t_0) < 1.$

We now show that 2) implies $2')$.  If $v$ is positive, bounded, and satisfies $\Delta v = \lambda v$ then the implication is clear.  If  the function is only non-zero then, tracing through the proof of the implications $2) \Rightarrow 3) \Rightarrow 1) \Rightarrow 1') \Rightarrow 2)$ given above, yields the existence of a positive, bounded function satisfying the equation $\Delta v = \lambda v$ making the implication, once more, immediate.

Therefore, we have shown the implications $1) \Leftrightarrow 1') \Rightarrow 2) \Leftrightarrow 2')$ and $2) \Rightarrow 3) \Rightarrow 1)$ which concludes the proof.
\end{proof}

\begin{definition}
A function $v$ satisfying $\Delta v = \lambda v$ for a constant $\lambda$ is called \emph{$\lambda$-harmonic}.  A function satisfying $\Delta v \leq \lambda v$ is called \emph{$\lambda$-subharmonic.}
\end{definition}

Therefore, stochastic incompleteness is equivalent to the existence, for every negative $\lambda$, of a bounded, positive, $\lambda$-harmonic (or $\lambda$-subharmonic) function and this criterion will be used to prove several results below.

\subsection{General Graphs}
For any vertex $x_0$ of $G$ we let $S_r(x_0)$, the \emph{sphere of radius $r$ about $x_0$}, denote the set of vertices that are exactly distance $r$ from $x_0$ and let $M_{x_0}(r)$ denote the maximum valence of the vertices in $S_r(x_0)$.  We now give a criterion for the stochastic completeness of a general graph.

\begin{theorem}\label{generalgraphs}
If $G$ is a graph with a vertex $x_0$ such that $M_{x_0}(r) = M(r) = \max_{x \in S_r(x_0)} m(x)$ satisfies
\[ \sum_{r=0}^\infty \frac{1}{M(r)} = \infty \]
then $G$ is stochastically complete.
\end{theorem}

\begin{remark}  This theorem should be compared with \cite{DodMat06}*{Theorem 2.10} which gives stochastic completeness under the assumption that the valence of the graph is bounded, that is, $m(x) \leq M$ for all vertices $x$.  The proof in \cite{DodMat06} uses a maximum principle argument to establish the uniqueness of bounded solutions of the heat equation with bounded initial data analogous to the case of a complete Riemannian manifold whose Ricci curvature is bounded from below \cite{Dod83}*{Theorem 2.2}. If we let $r(x)=d(x,x_0)$ then the argument in \cite{DodMat06} can be used to show stochastic completeness when $\Delta r(x) \geq -C$ for $C\geq0$ \cite{Web08}*{Theorem 4.13}.  Our approach is different in that we study $\lambda$-harmonic functions instead of
considering bounded solutions of the heat equation and our result shows that many graphs with $\Delta r \to -\infty$ as $r$ tends to infinity are stochastically complete.  For example, in Corollary \ref{models} below, we show that a model tree $T_n$ with branching number $n(i)$ is stochastically complete if and only if $\sum_{i=0}^{\infty} \frac{1}{n(i)} = \infty$ and, for such trees, $\Delta r(x) = 1 - n(i)$ for $x \in S_i(x_0)$.
\end{remark}

\begin{proof}
Let $v$ be a positive, $\lambda$-harmonic function on $G$ for $\lambda<0$.  By Theorem \ref{stochasticincomp} it suffices to show that $v$ is not bounded.  At $x_0$, $\Delta v(x_0) = \lambda v(x_0)$ yields $\sum_{x \sim x_0} v(x) = \big( m(x_0) - \lambda \big) v(x_0)$ implying that there exists $x_1 \sim x_0$ such that
\[ v(x_1) \geq \left(1-\frac{\lambda}{m(x_0)} \right) v(x_0). \]
By repeating the argument there exists a vertex $x_2 \in S_2(x_0)$ such that
\[ v(x_2) \geq \left(1-\frac{\lambda}{m(x_1)} \right) v(x_1) \geq \left(1-\frac{\lambda}{m(x_1)} \right) \left(1-\frac{\lambda}{m(x_0)} \right) v(x_0) .\]
Iterating the argument and noting that $\sum_{r=0}^\infty \frac{1}{M(r)} = \infty$ implies $\sum_{r=0}^\infty \frac{1}{m(x_r)} = \infty$ which, as $\lambda < 0$, is equivalent to $\prod_{r=0}^\infty \left(1 - \frac{\lambda}{m(x_r)}\right) = \infty$ \cite{Hil59}*{Theorem 8.6.1} the conclusion follows.
\end{proof}

In order to state an analogous criterion for the stochastic incompleteness of a graph we first introduce some more notation.  As before, we let $x_0$ be a fixed vertex of $G$ and, for $x \in S_r(x_0)$, let $m_{ \pm 1}(x) = | \{y \sim x \ | \ y \in S_{r \pm 1}(x_0) \} |$ denote the number of neighbors of $x$ in $S_{r \pm 1}(x_0)$, the next or previous sphere.  Let $\un{m}_{+1}(r)$ denote the minimum of $m_{+1}(x)$ for $x \in S_r(x_0)$ and $\ov{m}_{-1}(r)$ the maximum of $m_{-1}(x)$ for $x \in S_r(x_0)$.

\begin{theorem} \label{generalincompletegraphs}
If $G$ is a graph with a vertex $x_0$ such that \\
$\un{m}_{+1}(r) = \min_{x \in S_r(x_0)} m_{+1}(x)$ and $\ov{m}_{-1}(r) = \max_{x \in S_r(x_0)} m_{-1}(x)$ satisfy
\[ \sum_{r=1}^\infty \frac{  \ov{m}_{-1}(r)}  {\un{m}_{+1}(r)} < \infty \]
then $G$ is stochastically incomplete.
\end{theorem}

\begin{proof}
To prove the theorem we define a positive, bounded $\lambda$-subharmonic function depending only on the distance from $x_0$.  Let $v(0)= v(x_0) > 0$ be any positive constant.  For $x \in S_1(x_0)$ define
\begin{equation} \label{v(0)}
v(x)=v(1) = \left(1 - \frac{\lambda}{m(x_0)} \right) v(0)
\end{equation}
and, for $x \in S_{r+1}(x_0)$ where $r \geq 1$,
\begin{eqnarray}
\quad \quad v(x) = v(r+1) &=& \left( 1 + \frac{\ov{m}_{-1}(r) - \lambda}{\un{m}_{+1}(r)} \right) v(r) - \left( \frac{\ov{m}_{-1}(r)} {\un{m}_{+1}(r) } \right) v(r-1) \label{v(r+1)} \\
&=& \left( 1 - \frac{\lambda}{\un{m}_{+1}(r)} \right) v(r) + \left( \frac{\ov{m}_{-1}(r)} {\un{m}_{+1}(r)} \right) \big( v(r) - v(r-1) \big). \label{v(r+1)too}
\end{eqnarray}

From (\ref{v(0)}) it is clear that $v(1) > v(0)$ and, assuming $v(r) > v(r-1)$, it follows from (\ref{v(r+1)too}) that $v(r+1)>v(r)$.  Therefore, by induction, $v(r+1)>v(r)$ for all $r \geq0$ and $v$ is positive provided that $v(0)>0$.  Furthermore, using (\ref{v(r+1)}) repeatedly it follows that
\begin{eqnarray*}
v(r+1) &<&  \left( 1 + \frac{\ov{m}_{-1}(r) - \lambda}{\un{m}_{+1}(r)} \right) v(r) \\
&<& \prod_{i=0}^r  \left( 1 + \frac{\ov{m}_{-1}(i) - \lambda}{\un{m}_{+1}(i)} \right) v(0) < \prod_{i=0}^{\infty}  \left( 1 + \frac{\ov{m}_{-1}(i) - \lambda}{\un{m}_{+1}(i)} \right) v(0)
\end{eqnarray*}
which is finite from the assumption on $G$.  Therefore, $v$ is bounded.

It remains to show that $v$ is $\lambda$-subharmonic.  For $r=0$,
\begin{eqnarray*}
\Delta v(0) &=& m(x_0) \big( v(0) - v(1) \big) \\
&=& m(x_0) \left( v(0) - \left(1 - \frac{\lambda}{m(x_0)} \right) v(0) \right) = \lambda v(0).
\end{eqnarray*}
For $x \in S_r(x_0)$ where $r>0$, from (\ref{v(r+1)too}) and the fact that $v(r) - v(r-1)>0$, it follows that
\begin{eqnarray*}
\Delta v(x) &=& m_{+1}(x) \big( v(r) - v(r+1) \big) + m_{-1}(x) \big( v(r) - v(r-1) \big) \\
&=& \left( m_{+1}(x) \left( \frac{\lambda}{\un{m}_{+1}(r)} \right) \right) v(r)  \\
& & + \left( m_{-1}(x) - \ov{m}_{-1}(r) \left(\frac{ m_{+1}(x)}{\un{m}_{+1}(r) } \right) \right) \big( v(r) - v(r-1) \big) \\
&\leq&  \lambda v(r)
\end{eqnarray*}
thereby completing the proof.
\end{proof}

Theorem \ref{generalincompletegraphs} states that a graph will be stochastically incomplete if starting at a fixed vertex $x_0$ the number of edges leading away from $x_0$ is growing sufficiently rapidly in all directions from $x_0$.  The next result states that if we attach an arbitrary graph at $x_0$ to such a stochastically incomplete graph then the resulting graph will also be stochastically incomplete.  To make a precise statement we introduce some notation.  If $H$ is a subgraph of $G$ we let $H^C$ denote the complementary subgraph of $H$ in $G$.  Furthermore, we let, for $x_0 \in H$,
\[ \un{m}_{+1}^H (r) = \min_{\substack{x \in S_r(x_0) \\ x \in H}}  m_{+1}(x) \ \textrm{ and } \ \ov{m}_{-1}^H(r) = \max_{\substack{x \in S_r(x_0) \\ x \in H}}  m_{-1}(x) \]
so that the maximum and minimum are now taken over vertices in $H$.

\begin{theorem} \label{incompletedirection}
If $G = H \cup H^C$ is a graph with a vertex $x_0 \in H$ such that $\underline{m}_{\pm 1}^H(r)$ satisfy
\[  \sum_{r=1}^{\infty} \frac{\ov{m}_{-1}^H(r) }   {\underline{m}_{+1}^H(r)} < \infty \]
and $x_0$ is the only vertex of $H$ that has a neighbor in $H^C$ then  $G$ is stochastically incomplete.
\end{theorem}

\begin{proof}
The proof will use the following general statement which states that there always exists a positive, bounded function which is $\lambda$-harmonic except at one vertex.

\begin{lemma}\label{exception}
Let $G$ be a graph with $x_0 \in G$.  For every $\lambda<0$ there exists a function $v$ satisfying $v(x_0)=1$ with $0 < v(x) < 1$ and $\Delta v(x)=\lambda v(x)$ for $x \not=x_0$.
\end{lemma}

Assuming, for now, Lemma \ref{exception} we prove Theorem \ref{incompletedirection} by first defining a positive, bounded function $v$ which is $\lambda$-harmonic for all vertices $x \in H^C$ with $v(x_0)=w(0)=1$.  We then make $v$ $\lambda$-harmonic at $x_0$ by letting, for all vertices $x \in S_1(x_0) \subset H$,
\[ v(x) = w(1) = \left( \frac{1}{m_{+1}^H(0)} \right) \left( m(x_0) - \lambda - \sum_{\substack{x \sim x_0 \\ x \in H^C }} v(x) \right). \]
For $x \in S_{r+1}(x_0) \subset H$ let
\[ v(x) = w(r+1) = \left( 1 + \frac{\ov{m}_{-1}^H(r) - \lambda}{\un{m}_{+1}^H(r)} \right) w(r) - \left( \frac{\ov{m}_{-1}^H(r)} {\un{m}_{+1}^H(r) } \right) w(r-1)  \]
to define $v$, a positive, bounded, $\lambda$-subharmonic function depending on the distance from $x_0$, on the rest of $H$ as in the proof of Theorem \ref{generalincompletegraphs}.
\end{proof}

We now prove Lemma \ref{exception}:
\begin{proof}
Let $B_r=B_r(x_0)$ and observe that, for $\lambda < 0$, there exists a unique solution of the following system of equations:
\begin{equation}\label{system of equations}
\left\{ \begin{array}{ll} \Delta v_r(x) - \lambda v_r(x) = 0 & \textrm{ for } x \in \textrm{int } B_r \setminus \{x_0\} \\
v_r(x_0)=1 \\
v_r(x) = 0 & \textrm{ for } x \in \partial B_r.
\end{array} \right.
\end{equation}
To see this consider the associated homogeneous system obtained by replacing the second equation in (\ref{system of equations}) by $v_r(x_0)=0.$  We will show that the homogeneous system has only trivial solutions.  Let $w$ be a solution of the homogeneous system and suppose that $w$ is non-zero.  Then we can assume that $w(\hat{x})>0$ for some $\hat{x} \in $ int $B_r \setminus \{x_0\}$ and that $w(\hat{x})$ is a maximum for $w$.  It follows that $\Delta w(\hat{x})\geq 0$ while $\Delta w(\hat{x})=\lambda w(\hat{x})<0$.  The contradiction establishes the existence and uniqueness of a solution to (\ref{system of equations}).

One shows that $v_r$, the solution to (\ref{system of equations}), satisfies $0<v_r(x) < 1$ for $x \in $ int $B_r \setminus \{x_0\}$ by maximum principle arguments similar to the one used above.  Then, by extending each $v_r$ by 0 outside of $B_r$, it follows that $v_r \leq v_{r+1}$ so that $v_r \to v$ where $v$ satisfies $v(x_0)=1$, with $0<v(x)< 1$ and $\Delta v(x) = \lambda v(x)$ for all $x \not = x_0$.
\end{proof}

\begin{remark}
Lemma \ref{exception} can be generalized to obtain a function $v$ which is not $\lambda$-harmonic at finitely many vertices.  It can also be applied finitely many times to obtain a more general result then the one presented in Theorem \ref{incompletedirection}.  Specifically, one can attach to $H$ finitely many disjoint graphs at arbitrary vertices of $H$ and the resulting graph will be stochastically incomplete.  However, the stronger result that if $G$ contains a stochastically incomplete subgraph then $G$ is stochastically incomplete is not, in general, true.  For example, one can start with $T_n$, a stochastically incomplete model tree (see next subsection for the definition) and, by attaching infinitely many stochastically complete trees (models whose branching number is equal to 1, for example) at each vertex of $T_n$, construct a tree which contains $T_n$ and is stochastically complete.  See \cite{Woj07} for details.  In general, in an unpublished paper, M. Keller showed that any stochastically incomplete graph is a subgraph of a stochastically complete graph constructed in this manner but the number of attached graphs must be infinite.
\end{remark}

\subsection{Model Trees}
We will now show that the characterizations of stochastic completeness and incompleteness in terms of the growth of the valence on spheres given in Theorems \ref{generalgraphs} and \ref{generalincompletegraphs} are optimal by considering a specific family of trees.  Suppose that a tree contains a vertex $x_0$, referred to from now on as the \emph{root} for the tree, such that the valence at any other vertex only depends on the distance from $x_0$. Therefore, for all $x \in S_r(x_0)$, $m(x)=m(r)$.  We will let $n(0)=m(x_0)$ and, for $r>0$, $n(r)=m(r)-1$ denote the \emph{branching number} of such a tree by which we mean the number of edges on a sphere of radius $r$ about $x_0$ leading away from $x_0$.  We call such trees \emph{model} and denote them throughout by $T_n$.  As a consequence of previous results it follows that:

\begin{corollary} \label{models}
$T_n$ is stochastically complete if and only if
\[\sum_{r=0}^{\infty} \frac{1}{n(r)} = \infty. \]
\end{corollary}
\begin{proof}
Since, for a model tree $T_n$, $n(r)= \underline{m}_{+1}(r) = M(r) - 1$ and $\ov{m}_{-1}(r)=1$ for all $r>0$, Corollary \ref{models} follows by applying Theorems \ref{generalgraphs} and \ref{generalincompletegraphs}.
\end{proof}

\subsection{Heat Kernel Comparison}
In this subsection we will prove two inequalities comparing the heat kernel on a model tree and the heat kernel on a general graph.  For the Riemannian case see \cite{CheYau81}*{Theorem 3.1}.  Throughout, we will denote the heat kernel on the model tree by $\rho_t(x,y)$ to distinguish it from $p_t(x,y)$, the heat kernel on a general graph.  In order to prove our results we will need the following two lemmas concerning the heat kernel on model trees.

\begin{lemma}
Let $T_n$ be a model tree with root vertex $x_0$ and heat kernel $\rho_t(x,y)$.  Then, for all vertices $x \in S_r(x_0)$,
\[ \rho_t(x_0,x) = \rho_t(r). \]
\end{lemma}
\begin{proof}
This is clear from the fact that the Dirichlet heat kernels on $B_R(x_0)$ are defined by $\rho_t^R(x_0,x)= \la e^{-t\Delta_R} \delta_{x_0}, \delta_{x} \ra = \delta_{x_0}(x) - t \Delta_R(x_0,x) + \frac{t^2}{2} \Delta_R^2(x_0,x) - \dots$ where
\[ \Delta_R(x_0,x) = \Delta_R \delta_{x_0}(x) = \left\{  \begin{array}{lll}
 n(0) & \textrm{ if } x = x_0 \\
 - 1 & \textrm{ if } x \in S_1(x_0) \\
 0  & \textrm{ otherwise }
 \end{array} \right. \]
and $\Delta_R^{m+n}(x_0,x) = \sum_{y \in B_R} \Delta_R^m(x_0,y) \Delta_R^n(y,x).$
\end{proof}

\begin{lemma}\label{modelheatkernels}
Let $T_n$ be a model tree with root vertex $x_0$ and heat kernel $  \rho_t(r) = \rho_t(x_0,x)$ for $x \in S_r(x_0)$.  Then, for $r \geq 0$ and $t\geq0$,
\[\rho_t(r)\geq \rho_t(r+1). \]

\end{lemma}
\begin{proof}
The statement is clearly true for $t=0$.  Consider $\rho_t^R(r)$ on $B_R(x_0) \times [0,T]$ as above.  We claim that $\rho_t^R(r) > \rho_t^R(r+1)$ for all $r\geq0$ and $t > 0$.  For the case of $r=0$ observe that the eigenfunction expansion of $\rho_t^R(r)$, statement 5) in Proposition \ref{Dirichletheatkernels}, implies that $\pdt \rho_t^R(0) < 0$ since $\{ \phi_i^R(x) \}_{i=0}^{k(R)} $ forms an orthonormal basis for $C(B_R, \partial B_R)$ and, as such, $\phi_i^R(x_0)$ cannot be zero for all $i$.  Therefore,
\[ \Delta \rho_t^R(0) = n(0) \big( \rho_t^R(0) - \rho_t^R(1) \big) > 0 \]
or $\rho_t^R(0) > \rho_t^R(1).$  Also, $\rho_t^R(R-1) > \rho_t^R(R)$ since $\rho_t^R(R-1)>0$ while $\rho_t^R(R)=0$ for $t>0$.

Consider now the function
\[ \varphi(t) = \min_{i<j} \big( \rho_t^R(i) - \rho_t^R(j) \big). \]
The claim is that $\varphi(t) > 0$ for $t > 0$.  Suppose not, that is, suppose that there exists $t_0 > 0$ such that $\varphi(t_0) \leq 0$.  Therefore, there exist $i_0 < j_0$ such that $\rho_{t_0}^R(i_0) \leq \rho_{t_0}^R(j_0)$.  We can then assume that $\rho_{t_0}^R(r)$ has a global minimum at $r=i_0$ and a global maximum at $r=j_0$ for  $0< i < j < R$.  Then, from $\rho_t^R(0)>\rho_t^R(1)$ and $\rho_t^R(R-1) > \rho_t^R(R)$ shown above, we may assume that $\rho_{t_0}^R(i_0 -1)> \rho_{t_0}^R(i_0)$ and $\rho_{t_0}^R(j_0) > \rho_{t_0}^R(j_0+1)$ implying that
\[ \Delta \rho_{t_0}^R(i_0) < 0 \textrm{ and } \Delta \rho_{t_0}^R(j_0) > 0. \]
Therefore,
\[ \pdt \rho_{t_0}^R(i_0) > 0 \textrm{ and } \pdt \rho_{t_0}^R(j_0) < 0 \]
which implies that $\varphi'(t_0) > 0$.  It follows that there exists an $\epsilon >0$ such that for all $t \in (t_0 - \epsilon, t_0]$, $\varphi(t) < \varphi(t_0) \leq 0$.

Let $I$ denote the maximal interval contained in $[0,t_0]$ which contains $t_0$ and all $t< t_0$ for which $\varphi(t) \leq 0$.  It is clear from the continuity of $\varphi$ that $I$ is closed.  If $I = [a, t_0]$ for some $a>0$ then the argument above implies that there exists an $\epsilon>0$ such that all $t \in (a-\epsilon, a]$  are in $I$ contradicting the maximality of $I$.  If $I=[0,t_0]$ then it would follow that $\varphi(0)<0$ contradicting the fact that $\varphi(0)=0$.  In either case, we obtain a contradiction implying that $\varphi(t)>0$ for all $t>0$ and, therefore, that $\rho_t^R(r) > \rho_t^R(r+1)$.   Letting $R \to \infty$ completes the proof of the lemma.
\end{proof}

We are now ready to state and prove the following theorem
\begin{theorem} \label{comparisons}
Let $T_n$ denote a model tree with root vertex $x_0$ and heat kernel $\rho_t(r)=\rho_t(x_0,x)$ for $x \in S_r(x_0)$.  Let $G$ denote a graph with heat kernel $p_t(x,y)$.
\begin{enumerate}
\item[1)] If $G$ contains a vertex $x_0'$ such that, for all $x \in S_r(x_0') \subset G$, \\ $m_{+1}(x) \leq n(r)$ then
\[ \rho_t(r) \leq p_t(x_0',x) \ \textrm{ for all } x \in S_r(x_0') \subset G. \]
\item[2)] If $G$ contains a vertex $x_0'$ such that, for all $x \in S_r(x_0') \subset G$, \\ $n(r) \leq {m}_{+1}(x)$  and $m_{-1}(x) = 1$ then
\[ p_t(x_0',x) \leq \rho_t(r) \ \textrm{ for all } x \in S_r(x_0') \subset G. \]
\end{enumerate}
\end{theorem}
\begin{proof}
For both 1) and 2) we think of the heat kernel $\rho_t(x) = \rho_t(x_0,x)$ on $T_n$ as being defined for $x \in G$ by letting $\rho_t(x) = \rho_t(r)$ if $x \in S_r(x_0') \subset G$.  We let $\Delta_G$ and $\Delta_{T_n}$ denote the Laplacians on $G$ and $T_n$, respectively.  For 1) it follows from Lemma \ref{modelheatkernels} and from the assumption $m_{+1}(x) \leq n(r)$ that, for $x \in S_r(x_0') \subset G$, the Dirichlet heat kernel $\rho_t^R(x)$ on $B_R(x_0) \times [0,T]$ in $G \times [0,T]$ satisfies
\begin{eqnarray*}
\Delta_G \rho_t^R(x) &=& m_{+1}(x) \big( \rho_t^R(r)-\rho_t^R(r+1) \big) + m_{-1}(x) \big( \rho_t^R(r)-\rho_t^R(r-1) \big) \\
&\leq& n(r) \big(\rho_t^R(r)-\rho_t^R(r+1) \big) + \big( \rho_t^R(r) - \rho_t^R(r-1) \big) \\
&=& \Delta_{T_n} \rho_t^R(r)=-\pdt \rho_t^R(r).
\end{eqnarray*}
Letting $u^R(x,t)=p_t^R(x_0',x)- \rho_t^R(x)$ it follows that $\Delta_G u^R(x,t) + \pdt u^R(x,t) \geq 0$ on int $B_R(x_0)\times [0,T]$.  Lemma \ref{maxheatlem1} implies that
\[ \min_{B_R(x_0)\times [0,T]} u^R(x,t) = \min_{\substack{B_R(x_0) \times \{0\} \ \cup \\ \partial B_R(x_0) \times [0,T]}} u^R(x,t) = 0.\]
Therefore, $p_t^R(x_0',x) \geq \rho_t^R(r)$ for all $x \in S_r(x_0') \subset G$ and $p_t(x_0',x) \geq \rho_t(r)$ by letting $R \to \infty$.  The second statement is proved using the same techniques.
\end{proof}

\begin{remark}
The first result in Theorem \ref{comparisons} is an exact analogue of Theorem 3.1 in \cite{CheYau81}.  For the second result, the additional assumption that $m_{-1}(x)=1$ for all vertices $x$ implies that $G$ is obtained by starting with a tree and then allowing any two vertices on a sphere $S_r(x_0')$ to be connected by an edge.  In particular, for every vertex in such a graph there exists a unique shortest path connecting that vertex to $x_0'$.
\end{remark}

\section{Essential Spectrum}
\subsection{Bottom of the Spectrum}
We begin by recalling a characterization of the bottom of the spectrum of the Laplacian in terms of positive $\lambda$-harmonic functions.  Fix a vertex $x_0 \in G$ and let, as before, $B_r=B_r(x_0)$ with $C(B_r, \partial B_r)$ denoting those functions on $B_r$ which vanish on the boundary.   It can be shown \cite{Dod84}*{Lemma 1.9} that the smallest eigenvalue of the reduced
Laplacian $\Delta_r$ acting on $C(B_r,\partial B_r)$ is a simple eigenvalue given by
\[ \lambda_0(\Delta_r) = \min_{f \in C(B_r, \partial B_r) \setminus \{0\}} \frac{\la df, df \ra}{\la f,f \ra}. \]
Therefore, $\lambda_0(\Delta_r) \geq \lambda_0(\Delta_{r+1})$ and the bottom of the spectrum of the Laplacian is defined as the limit $\lambda_0(\Delta)=\lim_{r \to \infty}\lambda_0(\Delta_r)$.  One could also let $\lambda_0(\Delta)=\inf_{f \in C_0(V) \setminus \{0\}} \frac{\la \Delta f, f \ra}{\la f,f \ra}.$  We now state the following theorem which, in the Riemannian setting, is given in \cite{Sul86}*{Theorem 2.1} (see also \cite{FisSch80}*{Lemma 1}) and, in the case of the combinatorial Laplacian, in \cite{DodKar88}*{Proposition 1.5}.
\begin{theorem}
For every $\lambda \leq \lambda_0(\Delta)$ there exists a positive function $v$ satisfying $\Delta v = \lambda v$.  For every $\lambda > \lambda_0(\Delta)$ no such functions exist.
\end{theorem}
\begin{proof}
For $\lambda \leq \lambda_0(\Delta) \leq \lambda_0(\Delta_r)$ note that, as in the proof of the implication $2) \Rightarrow 2')$ of Theorem \ref{stochasticincomp}, there exists a function satisfying $\Delta v_r = \lambda v_r$ on the interior of $B_r$ such that $v_{r| \partial B_r}=1$.  It was shown there that $0 < v_r<1 $ on the interior of $B_r$. Let $w_r = \frac{v_r}{v_r(x_0)}$ to obtain a positive function which is $\lambda$-harmonic on the interior of $B_r$ and satisfies $w_r(x_0)=1$.   By a Harnack inequality argument the sequence $\{ w_r(x) \}_{r=0}^\infty$ is bounded for every fixed vertex therefore, by the diagonal process, we may find a subsequence which converges for all vertices (see \cite{DodKar88}*{Proposition 1.5} and \cite{Woj07} for details).  This gives the proof of the first statement.

For the second statement, if $v$ is a positive function satisfying $\Delta v= \lambda v$ then $u(x,t)=e^{-\lambda t}v(x)$ and $w(x,t) = \sum_{y \in B_r} p_t^r(x,y)v(y)$ both satisfy the heat equation on int $B_r \times [0,T]$.  Applying Lemma \ref{maxheatlem1} to $u-w$ implies that $u-w \geq 0$ on $B_r \times [0,T]$.  Therefore, $e^{-\lambda t}v(x) \geq \sum_{y \in B_r} p_t^r(x,y)v(y)$, which we write as
\begin{equation} \label{t to infty}
v(x) \geq e^{(\lambda - \lambda_0^r)t} e^{\lambda_0^r t} \sum_{y \in
B_r}p_t^r(x,y).
\end{equation}
Now, $p_t^r(x,y) = \sum_{i=0}^{k(r)} e^{-\lambda_i^r t} \phi_i^r(x) \phi_i^r(y)$ implies that
\[ \lim_{t \to \infty} e^{\lambda_0^r t} p_t^r(x,y) = \phi_0^r(x) \phi_0^r(y). \]
By applying Lemma \ref{maxheatlem2}, $\phi_0^r$ can be chosen so that $\phi_0^r(x)>0$ for $x$ in the interior of $B_r$.  Therefore, if $\lambda > \lambda_0^r=\lambda_0(\Delta_r)$  then the right hand side of (\ref{t to infty}) would go to $\infty$ as $t \to \infty$. Hence, $\lambda \leq \lambda_0(\Delta_r)$ for all $r$ implying $\lambda \leq \lambda_0(\Delta)$.
\end{proof}

\subsection{Lower Bounds}
We now use the approach in \cite{DodKar88} to prove a lower bound on the bottom of the spectrum under a curvature assumption on the graph.  In order to take advantage of a lower bound on the bottom of the spectrum in terms of Cheeger's constant proved in \cite{DodKen86} we have to introduce an operator related to the Laplacian $\Delta$.  Specifically, we denote by $\Delta_{bd}$ the \emph{bounded} or \emph{combinatorial} Laplacian which is given by
\[ \Delta_{bd} f(x) = \frac{1}{m(x)} \sum_{y \sim x} \big(f(x) - f(y) \big) = \frac{1}{m(x)} \Delta f(x). \]
This operator acts on the Hilbert space 
\[ \ell^2_{bd}(V) = \{ f:V \to \mathbf{R} \ | \ \sum_{x \in V} m(x) f(x)^2 < \infty \} \]
with inner product $\la f, g \ra_{bd} = \sum_{x \in V}m(x) f(x)g(x). $ $\Delta_{bd}$ is a self-adjoint, bounded operator with $|| \Delta_{bd} || \leq 2$.  Furthermore, using the technique in the proof of Theorem \ref{generalgraphs}, it can be shown that the heat kernel associated to $\Delta_{bd}$, that is, $e^{-t \Delta_{bd}}\delta_x(y)$, is stochastically complete for all graphs.  In particular, since the various characterizations of stochastic incompleteness given in Theorem \ref{stochasticincomp} hold for $\Delta_{bd}$ as well as for $\Delta$, bounded solutions of the heat equation involving $\Delta_{bd}$ are uniquely determined by initial conditions for any graph \cite{Woj07}.

The bottom of the spectrum of $\Delta_{bd}$ is, as for $\Delta$, given by
\[\lambda_0(\Delta_{bd}) = \inf_{f \in \ell^2(V)_{bd} \setminus \{ 0\}} \frac{\la \Delta_{bd} f, f \ra_{bd}}{\la f, f \ra_{bd}}.\]
For any finite subgraph $D$ of $G$ we let $L(\partial D) = | \{ y \sim x \ | \ x \in D, y \not \in D \} |$ denote the number of edges with exactly one vertex in $D$ and $A(D) = \sum_{x \in D} m(x)$.  \emph{Cheeger's constant} is then be defined as $\alpha(G) = \inf_{\substack{ D \subset G \\ D \textrm{ finite}}} \frac{L(\partial D)}{A(D)}.$  The Theorem in \cite{DodKen86} states that
\[ \lambda_0(\Delta_{bd}) \geq \frac{1}{2} \alpha^2(G). \]

In fact, the proof in \cite{DodKen86} applies in the following more general context. Let $A$ denote a finite subgraph of $G$ and let $\Delta_{bd, G \setminus A}$ denote the reduced bounded Laplacian which is equal to $\Delta_{bd}$ on the complement of $A$ and 0 on $A$.  $\Delta_{bd, G \setminus A}$ acts on the space of $\ell^2_{bd}$ functions which vanish on $A$.  If $\alpha(G \setminus A) = \inf_{\substack{ D \textrm{ finite} \\ D \cap A = \emptyset}} \frac{L(\partial D)}{A(D)}$ then the proof in \cite{DodKen86} gives
\begin{equation}\label{Cheeg}
\lambda_0(\Delta_{bd, G \setminus A}) \geq \frac{1}{2} \alpha^2(G \setminus A).
\end{equation}

As usual, we fix a vertex $x_0 \in G$ and let $r(x) = d(x,x_0)$.  We let $m_{+1}(x) = | \{y \sim x \ | \ r(y)=r(x)+1 \} |$ and $m_{-1}(x)= | \{ y \sim x \ | \ r(y)=r(x)-1 \} |$ denote the number of vertices that are 1 step further and 1 step closer to $x_0$ then is $x$ as before.  It follows by an easy calculation that $\frac{m_{+1}(x)-m_{-1}(x)}{m(x)} \geq c > 0$ if and only if $\Delta_{bd}r(x) \leq -c < 0$.

\begin{lemma}\cite{DodKar88}*{Lemma 1.15}\label{Cheegbound}
If $A$ is a finite subgraph of $G$ and for $x \in G \setminus A$ 
\begin{equation} \label{curvature}
\frac{m_{+1}(x) - m_{-1}(x)}{m(x)} \geq  c >0
\end{equation}
then $\alpha(G \setminus A) \geq c.$
\end{lemma}
\begin{proof}
Lemma \ref{Green's} implies that
\[ \left| \sum_{x \in D} \Delta_{bd}r(x)m(x) \right| = \left| \sum_{x \in D} \Delta r(x) \right| =\left| \sum_{\substack{ x \in \partial D \\ y \sim x, y \not \in D}} \big( r(x) - r(y) \big) \right| \leq L(\partial D). \]
On the other hand, a calculation and (\ref{curvature}) imply
\[ \left| \sum_{x \in D} \Delta_{bd}r(x)m(x) \right| =  \sum_{x \in D}|m_{+1}(x)-m_{-1}(x)| \geq c \sum_{x \in D} m(x) = c A(D) \]
so that $c \leq \frac{L(\partial D)}{A(D)}$ for all finite subgraphs $D$.  Taking the infimum over all finite subgraphs disjoint from $A$ implies that $\alpha(G\setminus A) \geq c$.
\end{proof}

Combining (\ref{Cheeg}) and Proposition \ref{Cheegbound} we can now state and prove the following theorem:
\begin{theorem}\label{lowerbounds}
If $A$ is a finite subgraph of $G$ and for $x \in G \setminus A$ 
\[ \frac{m_{+1}(x) - m_{-1}(x)}{m(x)} \geq  c >0 \]
then $\lambda_0(\Delta_{bd, G \setminus A}) \geq \frac{c^2}{2}.$  If, in addition,  $m(x) \geq m$ for  $x \in G \setminus A$  then
\begin{equation}\label{lambda0}
\lambda_0(\Delta_{G \setminus A}) \geq \frac{mc^2}{2}.
\end{equation}
\end{theorem}
\begin{proof}
The first statement is a direct consequence of Lemma \ref{Cheegbound} and (\ref{Cheeg}).  To obtain the second statement note that, for any finitely supported function $f$, $\la \Delta_{bd} f, f \ra_{bd} = \la \Delta f, f \ra$ while, if $m(x) \geq m$, then $\la f, f \ra_{bd} \geq m \la f, f \ra$.  This estimate and the first statement in the theorem then combine to give
\[ \frac{c^2}{2} \leq \lambda_0(\Delta_{bd}) \leq \frac{\la \Delta_{bd} f, f \ra_{bd}}{\la f, f \ra_{bd}} \leq  \frac{\la \Delta f, f \ra}{m \la f, f\ra} \]
which, after taking the infimum over finitely supported functions which vanish on $A$, implies (\ref{lambda0}).
\end{proof}

\subsection{Essential Spectrum}
We now use Theorem \ref{lowerbounds} to give a condition under which the Laplacian $\Delta$ has empty essential spectrum.  In \cite{Fuj96}*{Theorem 1} it was shown that the essential spectrum of the bounded Laplacian $\Delta_{bd}$ shrinks to a point if and only if Cheeger's constant at infinity is equal to 1.  Using this result it is shown in \cite{Kel07}*{Theorem 2} that, if Cheeger's constant at infinity is positive, $\Delta$ has empty essential spectrum if and only if the graph is branching rapidly.

We first recall a characterization of the essential spectrum. It can be shown that $\lambda$ is in the essential spectrum of $\Delta$, denoted $\lambda \in \sigma_{ess}(\Delta)$, if and only if there exists a sequence $f_i$ in the domain of $\Delta$ which is orthonormal and satisfies $\Delta f_i - \lambda f_i \to 0$ \cite{ReeSim72}*{Theorem VII.12 and remarks following Theorem VIII.6}.  Here, the domain of $\Delta$ consists of $f \in \ell^2(V)$
such that $\Delta f \in \ell^2(V)$.  We denote by $\Delta_{G \setminus B_r}$ the reduced Laplacian which is equal to  $\Delta$ on the complement of $B_r$ and 0 on $B_r$ as in the previous subsection.  Using the characterization of the essential spectrum mentioned above one shows that
\begin{lemma}\citelist{ \cite{DonLi79}*{Proposition 2.1} \cite{Fuj96}*{Lemma 1} } \label{ess}
\[\sigma_{ess}(\Delta) = \sigma_{ess}(\Delta_{G \setminus B_r}).\]
\end{lemma}

Let $\underline{m}_c(r) = \inf_{x \in G \setminus B_r} m(x)$ denote the infimum of the valence of vertices outside the ball of radius $r$ about $x_0$.  Combining Lemma \ref{ess} and Theorem \ref{lowerbounds} we obtain the following criterion
\begin{theorem}
If for all vertices $x$ of $G$
\[ \frac{m_{+1}(x) - m_{-1}(x)}{m(x)} \geq  c >0 \]
and if $\underline{m}_c(r) \to \infty$ as $r \to \infty$ then $\sigma_{ess}(\Delta) = \emptyset$.
\end{theorem}
\begin{proof}
Applying (\ref{lambda0}) in Theorem \ref{lowerbounds} to $\lambda_0(\Delta_{G \setminus B_r})$ implies that  $\lambda_0(\Delta_{G \setminus B_r}) \to \infty$ as $r \to \infty$ since $\underline{m}_c(r) \to \infty$.  Meanwhile, by Lemma \ref{ess}, the essential spectrum of $\Delta$ remains unchanged after the removal of $B_r$.  Since the essential spectrum is a subset of the spectrum the result follows.
\end{proof}

\end{document}